\documentclass[12pt]{amsart}
\usepackage{amsmath, amssymb}
\usepackage{amscd}
\usepackage{verbatim}
\newtheorem{definition}{Definition}[section]
\newtheorem{theorem}[definition]{Theorem}
\newtheorem{lemma}[definition]{Lemma}
\newtheorem{corollary}[definition]{Corollary}
\newtheorem{proposition}[definition]{Proposition}
\theoremstyle{definition}
\newtheorem{remark}[definition]{Remark}

\newtheorem{example}[definition]{Example}

\newcommand{\la}{\left\langle}
\newcommand{\ra}{\right\rangle}

\newcommand{\clR}{\mathcal{R}}

\newcommand{\clL}{\mathcal{L}}

\newcommand{\bbN}{\mathbb{N}}
\newcommand{\bbC}{\mathbb{C}}
\newcommand{\clP}{\mathcal{P}}

\newcommand{\di}{\int^\oplus}
\newcommand{\dmu}{d\mu (p)}

\newcommand{\acomm}{{\pi(A)}^{\prime}}
\newcommand{\aocom}{{\pi_\omega(A)}^{\prime}}
\newcommand{\adcom}{{\pi(A)}^{\prime \prime}}

\newcommand{\cfinfty}{{C^*(\mathbb{F}_\infty)}}

\begin{document}
	

\title[Decomposable Weak Expectations]{Decomposable Weak Expectations}

\author[A.~Bhattacharya]{Angshuman Bhattacharya}
\address{Department of Mathematics, IISER Bhopal, MP 462066, India}
\email{angshu@iiserb.ac.in}

\author[C. J.~Kulkarni]{Chaitanya J. Kulkarni}
\address{Department of Mathematics, IISER Bhopal, MP 462066, India}
\email{kulkarni18@iiserb.ac.in}

\keywords{Weak Expectation Property, decomposition of representations, extremal decompositions}
\subjclass[2010]{Primary (2020) 46L45, 47B65, 47C15; Secondary (2020) 46L05, 47L10, 47B02}


\begin{abstract}
In this article we define a special class of weak expectations for a representation of a separable unital C*-algebra, called decomposable weak expectation. We give necessary and sufficient conditions for such kind of weak expectations to exist for a given representation. Then we define decomposable measures on the state space of a C*-algebra and show that the GNS representation of a state admits a decomposable weak expectation if and only if there is a decomposable measure on the state space. Further we give an example of a decomposable weak expectation. 
\end{abstract}

\maketitle


All C*-algebras and representations considered in this article are unital. 	

\section{Introduction}

A \textit{weak expectation} of a representation $\pi$ of a C*-algebra $A$ on a Hilbert space $H$ is a unital completely positive map $$\Phi: B(H) \rightarrow {\pi(A)}^{\prime\prime}$$ such that $\Phi(\pi (a))= \pi (a)$ for all $a\in A$ where ${\pi(A)}^{\prime\prime}$ denotes the double commutant of $\pi(A)$ in $B(H)$. A C*-algebra $A$ is said to have the \textit{weak expectation property} if $A$ admits a weak expectation for every \textit{faithful} representation $\pi: A\rightarrow B(H)$. The \textit{weak expectation property} of a C*-algebra was introduced by Lance \cite{L} in his characterization of nuclearity of discrete group C*-algebras. 

In 1993, Kirchberg \cite{Kir93} connected the weak expectation property to the all important \textit{Connes embedding conjecture} besides establishing a tensor product characterization of the same. The weak expectation property has been investigated is some detail, albeit sporadically, by several authors thereafter. An operator system approach to studying the weak expectation property was initiated in \cite{FP12, FKP13}, while several characterizations of the weak expectation property were formulated using operator system tensor products in \cite{FKPT18}. More generally, the weak expectation property has been studied in the category of operator systems \cite{KPTT13} and an intrinsic order theoretic characterization (in the operator system category) was found in \cite{Lup18}. Prior to these developments, the weak expectation property for C*-algebras had been characterized using operator spaces in \cite{JLM99}. The reader is directed to \cite{BO, Oz04} for a comprehensive introduction to these and related topics and \cite{BF} for an elementary permanence property. Finally, for a more recent and detailed study on the weak expectation property see \cite{Pis20} and the references therein.

In contrast to the aforementioned literature on the weak expectation property, this article is an attempt to investigate the \textit{existence of a weak expectation of a given representation} (unital, on a separable Hilbert space) of a unital separable C*-algebra. While the answer to such, in the most general context is perhaps unlikely, we settle the question for a more tractable instance, using decompostion theory of representations of a separable, unital (in general non-type I) C*-algebra. The weak expectations obtained for these instances are what we call \textit{decomposable} weak expectations (see Definition \ref{dwe} for details).

We give a brief overview of what follows in this article. After recalling a few necessary notions from decomposition theory of representations of a separable C*-algebra $A$ in section 2, we establish a simple reduction of the general framework which is relevant to our analysis. Then we introduce a special class of maximal abelian subalgebras (MASA) in the commutant of a representation. We call these ``decomposable" MASAs. We shall see later that such a MASA is easy to construct in general. Following the introduction of decomposable MASAs, we define decomposable weak expectations in section 3. Further, in the same section, we give necessary and sufficient conditions for such weak expectations to exist for a given representation. It is perhaps worth noting here that, our proof is constructive and it gives a method to construct a family of such decomposable weak expectations. In section 4, we discuss an application of our decomposition theory approach to construct weak expectations. More specifically, we analyse the case of the GNS representation of a state $\omega$ and when does it admit a decomposable weak expectation. It turns out that such a situation can occur if and only if the state in question admits a \textit{maximal orthogonal decomposable measure} on the state space $S(A)$ of the C*-algebra with $\omega$ as the barycenter of the measure. The class of orthogonal \textit{decomposable} measures on the state space $S(A)$ is a sub-class of orthogonal measures on $S(A)$, i.e. every orthogonal measure need not be a decomposable orthogonal measure. In the final section of this article (section 5) we give an example of a non-degenerate (not faithful, of course) representation of a separable C*-algebra, namely the full group C*-algebra $\cfinfty$ of the free group on countably infinite generators, which admit a decomposable weak expectation and we construct such a weak expectation to conclude the article.

In a forthcoming article we shall study the space of weak expectations of a represented C*-algebra as a compact affine space from a barycentric perspective.

\section{Preliminaries}

In this section, we briefly recall some definitions and results from decomposition theory of representations of C*-algebras first. For a detailed introduction to this subject see \cite{OB1, DixC, DixV, KR1, KR2}. 

Let $(X,\mu)$ be a standard measure space \cite[Section 4.4.1]{OB1}. Suppose $\{H_p\}_{p\in X}$ denotes a \textit{measurable family} of (separable) Hilbert spaces \cite[Definition 4.4.1B]{OB1}. The direct integral Hilbert space of the family $\{H_p\}_{p\in X}$ is denoted by $$H=\di_X H_p~d\mu(p).$$ The abelian von-Neumann algebra $L^\infty (X, \mu)$ acts as multiplication operators on the direct integral Hilbert space $H$ and are called \textit{diagonalizable operators} on $H$. For a measurable, essentially bounded family of operators $\{T_p\}_{p\in X}$ (see \cite[Section 4.4.1]{OB1}), denote the direct integral of $\{T_p\}_{p\in X}$ by $$\di_X T_p~d\mu (p).$$ The bounded operators of this form on the direct integral Hilbert space $H$ are called \textit{decomposable operators}. The collection of decomposable operators form a von-Neumann algebra and the following theorem is a well known fact:

\begin{theorem} \cite[Theorem 14.1.10]{KR2} \label{RL}
If $H$ is the direct integral of $\{H_p\}_{p\in X}$ over $(X,\mu)$, the set $\clR$ of decomposable operators is a von-Neumann algebra with abelian commutant given by the algebra of diagonalizable operators  on $H$ i.e $\clR^{'}=L^\infty(X,\mu)$.
\end{theorem} 

Let $A$ be a unital separable C*-algebra and let $\{\pi_p : A\rightarrow B(H_p)\}_{p\in X}$ be a family of representations of $A$ on the measurable family of Hilbert spaces $\{H_p\}_{p\in X}$. The family $\{\pi_p\}_{p\in X}$ is called measurable if: for all $a\in A$, $\{\pi_p(a)\}_{p\in X}$ is a measurable family of essentially bounded operators. Therefore, for all $a\in A$: $$\di_X \pi_p (a)~d\mu(p) \in \clR.$$

Now, let $A$ be a separable unital C*-algebra and $\pi:A\rightarrow B(H)$ be a faithful non-degenerate representation of $A$ on a separable Hilbert space $H$ (not to be confused with the direct integral Hilbert space mentioned above, at the moment). Let $\acomm$ denote the commutant of $\pi(A)$ in $B(H)$. Let $\clL \subset \acomm$ be an abelian von-Neumann subalgebra of $\acomm$. The \textit{direct integral decomposition} of $\pi$ with respect to $\clL$ is given by the following fundamental theorem for the spatial decomposition of representations:

\begin{theorem}\cite[Theorem 4.4.7]{OB1} \label{dint}
Let $A$ be a separable C*-algebra, $\pi$ a non-degenerate representation of $A$ on a separable Hilbert space $H$, and $\clL$ be an abelian von-Neumann subalgebra of $\acomm$. It follows that there exists a standard measure space $X$, a positive bounded measure $\mu$ on $X$, a measurable family of Hilbert spaces $\{H_p\}_{p\in X}$, a measurable family of representations $\{\pi_p\}_{p\in X}$ on $\{H_p\}_{p\in X}$ and a unitary map $$U: H \rightarrow \di_X H_p~d\mu(p)$$ such that $U\clL U^*$ is just the set of diagonalizable operators on $$\di_X H_p~d\mu(p)$$ and $$U\pi(a) U^*=\di_X \pi_p (a)~d\mu(p)$$ for all $a\in A$.
\end{theorem}

An important consequence of the aforementioned theorem, which is of particular relevance to this work, is the case of extremal decompositions given below:

\begin{corollary}\cite[Corollary 4.4.8]{OB1} \label{max}
Let $A$ be a separable C*-algebra, $\pi$ a non-degenerate representation of $A$ on a separable Hilbert space $H$, and $\clL$ be an abelian von-Neumann subalgebra of $\acomm$ and $$\pi=\di_X \pi_p~d\mu(p)$$ be the decomposition of $\pi$ w.r.t. $\clL$. Then the following two statements are equivalent:
\begin{enumerate}
\item $\pi_p$ is irreducible for $\mu$-almost all $p\in X$.
\item $\clL$ is maximal abelian in $\acomm$.
\end{enumerate}
\end{corollary}

Next, we briefly recall the barycentric disintegration of a state $\omega$ of a separable unital C*-algebra $A$ and the associated spatial disintegration of the GNS representation $\pi_{\omega}$ of the state, due to Effros. For this, we quickly recall the concept of an \textit{orthogonal measure} on the state space $S(A)$ of $A$.

A pair of positive linear functionals $\omega_1, \omega_2 \in S(A)$ are said to be \textit{orthogonal} if the positive linear functional $\omega=\omega_1 + \omega_2$ be such that $H_{\omega}= H_{\omega_1}\oplus H_{\omega_2}$, $\pi_{\omega}=\pi_{\omega_1} \oplus \pi_{\omega_2}$ and $\xi_\omega = \xi_{\omega_1} \oplus \xi_{\omega_2}$, where $(\pi_\omega, H_\omega, \xi_\omega)$ denotes the GNS tuple for $\omega$. This is denoted by $\omega_1 \perp \omega_2$.

\begin{definition}\cite[Definition 4.1.20]{OB1}
A positive regular Borel measure $\mu$ on $S(A)$ is said to be orthogonal if, for any Borel set $E\subset S(A)$, one has $$\int_E \omega^\prime d\mu (\omega^\prime) \perp \int_{S(A)\backslash E} \omega^\prime d\mu (\omega^\prime)$$ where, the integral denotes the barycenter of $\mu$ on the respective sets of integration.
\end{definition}

See also \cite[Theorem 4.1.25]{OB1} for the correspondence between orthogonal measures with barycenter $\omega$ and the abelian von-Neumann subalgebras of $\aocom$. The following theorem (stated in the separable context, for orthogonal measures as relevant to our work) is due to Effros and connects the disintegration theory of representations to the barycentric decomposition of states of a (separable) unital C*-algebra in terms of its GNS representation:

\begin{theorem}\cite[Theorem 4.4.9 (Effros)]{OB1} \label{eff}
Let $A$ be a unital separable C*-algebra and $\mu$ be an orthogonal probability measure on $X=S(A)$ with barycenter $\omega$. Let $$\pi_\mu=\di_X \pi_{\omega^\prime}~d\mu(\omega^\prime)$$ be the direct integral representation of $A$ on $$H_\mu=\di_X H_{\omega^\prime}~d\mu(\omega^\prime).$$ Then $\pi_\omega$ is unitarily equivalent to $\pi_\mu$ and the direct integral decomposition of $\pi_\omega$ with respect to the abelian von-Neumann subalgebra $\clL_\mu$ of $\aocom$ is given by $$\pi_\omega=\di_X \pi_{\omega^\prime}~d\mu(\omega^\prime).$$ Here $\clL_\mu$ is the unique abelian von-Neumann subalgebra of $\aocom$ corresponding to the orthogonal measure $\mu$ \cite[Theorem 4.1.25]{OB1}.
\end{theorem}

Finally to conclude this section, we point out a simplification of our framework for studying weak expectations of a given representation (using decomposition theory), which is relevant to our context.

Let $\pi: A \rightarrow B(H)$ be a representation of a unital, separable C*-algebra on a separable Hilbert space $H$. For an abelian von-Neumann subalgebra $\clL \subset \acomm$ we note that the decomposable operators $\clR$ w.r.t. $\clL$ is an injective von-Neumann algebra by Theorem \ref{RL}. Further, by Theorem \ref{dint}, we have $(\clL_p)^{\prime} = B(H_p)$ a.e. and by \cite[Proposition 14.1.24]{KR2} we get: $$\clR=\di_X B(H_p)~d\mu(p).$$ Note that, since $\pi(A) \subset \clR$, we have $\adcom \subset \clR$, i.e. $\adcom$ is a decomposable von-Neumann algebra. In general, $\adcom\neq \clR$.

We will use an easy reduction of our framework, to study weak expectations, as given below:

\begin{proposition} \label{red}
Let $\pi: A \rightarrow B(H)$ be a representation of a unital, separable C*-algebra on a separable Hilbert space $H$. Let $\clL \subset \acomm$ an abelian von-Neumann subalgebra and $\clR$ be the decomposable operators w.r.t. $\clL$. Then $\pi$ admits a weak expectation if and only if there exist a unital completely postive map $\Phi: \clR \rightarrow \adcom$ such that $\Phi (\pi(a))=\pi(a)$ for all $a\in A$. 
\end{proposition}

\begin{proof}
For simplicity of nomenclature we call the map $\Phi$ a weak expectation of $\pi$ on $\clR$. The proof follows from the injectivity of $\clR$.  
\end{proof}

\section{Decomposable Weak Expectations}

In the rest of this article, for a given representation $\pi:A\rightarrow B(H)$ of a unital separable C*-algebra on a separable Hilbert space $H$, we will only consider the disintegration of $\pi$ w.r.t. a \textit{maximal abelian} von-Neumann subalgebra $\clL$ of $\acomm$. In that case, by Corollary \ref{max} we get $$\pi=\di_X\pi_p~\dmu,$$ where $\pi_p$ is \textit{irreducible} a.e. on $X$. Henceforth, we shall omit the differential $\dmu$ and simply write $\pi=\di_X\pi_p$ and elsewhere for direct integrals, to ease notation. Before embarking on the notion of decomposable weak expectations, we need some preparations.

Let $(X, \mu)$ be the standard measure space obtained by Theorem \ref{dint} for a representation $\pi:A\rightarrow B(H)$ corresponding to a given maximal abelian von-Neumann subalgebra $\clL \subset \acomm$ and $\pi=\di_X\pi_p$. Then outside a set of measure zero (say $X_0$), we have $\pi_p$ to be irreducible for $p\in X\backslash X_0$. Next, we define a natural equivalence relation on $X\backslash X_0$. For $p,q \in X\backslash X_0$, say $p\sim q$ if $\pi_p$ is unitarily equivalent to $\pi_q$ (which is possible by virtue of extremal decomposition of $\pi$ w.r.t. maximal abelian $\clL$). Denote the equivalence class of $p$ in $X\backslash X_0$ by $E_p$. Then we have the set theoretic partition $$X\backslash X_0=\coprod_{p \in X\backslash X_0} E_p.$$ We characterize when this partition is measurable, in the proposition below. This is of primary importance to our analysis that follows thereafter. But before that we need some notation.

Let $\Lambda$ denote the set of distinct classes of irreducible representations of $A$ which appear in the disintegration of $\pi$ with respect to $\clL$. Let $\{\rho\}$ denote an exhaustive tuple of representative from each class in $\Lambda$ on Hilbert spaces $\{H_\rho\}$. Further let $\{h_\rho\}_{\rho\in\{\rho\}}$ denote a tuple of unit vectors from $\{H_\rho\}$.

\begin{proposition} \label{mbl}
Let $E_p$ be the equivalence class of $p\in X\backslash X_0$ as defined above.  Then $\coprod E_p$ is a countable measurable partition if and only if there are countably many elements in $\{\rho\}$ and for any $\{\rho\}$ and any tuple of unit vectors $\{h_\rho\}_{\rho\in\{\rho\}}$ there exist a family of unitaries $$\{V^*_{\rho, p}: H_\rho \rightarrow H_p\}_{p\in E_p, \rho \in \{\rho\}}$$ such that the function defined on $X\backslash X_0$ by $p\mapsto V^*_{\rho, p} h_\rho$ is a measurable vector and $V_{\rho, p} \pi_p V_{\rho, p}^*=\rho$ for all $p\in E_p$.
\end{proposition}

\begin{proof}
First we prove the sufficiency condition for the partition to be measurable. Let $\{\rho\}$ and $\{h_\rho\}_{\rho\in\{\rho\}}$ be as described in the proposition and assume that the family of unitaries $\{V^*_{\rho, p}: H_\rho \rightarrow H_p\}_{p\in E_p, \rho \in \{\rho\}}$ exist with the properties as mentioned in the statement above. It is enough to prove that, for any $p\in X\backslash X_0$, the equivalence class $E_p$ is measurable. Consider the measurable vector $$h= \di_{X\backslash X_0} h_p = \di_{X\backslash X_0} V^*_{\rho, p} h_\rho \in \di_{X\backslash X_0} H_p.$$ Note that $\|h_p\|=1$ for all $p\in X\backslash X_0$ and $\| h\|=\mu(X)^{\frac{1}{2}}< \infty$.
	
Let $\{a_1, a_2,\ldots\}$ be a countable dense subset of the C*-algebra $A$ and $h= \di_{X\backslash X_0} h_p\in \di_X H_p$ be as described above. For $n\in\bbN$, define the function $f_n : X\backslash X_0 \rightarrow \bbC$ given by $$f_n (p)=\langle \pi_p(a^*_na_n)h_p, h_p \rangle = \| \pi_p(a_n)h_p \|^2.$$ Clearly $\{f_n\}_{n\in \bbN}$ are all measurable functions. Pick and fix $\rho\in\{\rho\}$ such that $\rho$ is equivalent to any $\pi_p$ for $p\in E_p$ and $h_\rho$ be the corresponding unit vector from $\{h_\rho\}$. For $m, n\in \bbN$ define the measurable sets: $$F_{\rho, h_\rho, m, n}:=\{ p\in X\backslash X_0 : |~ \|\rho(a_n)h_\rho \|^2 -  f_n(p)~ | < 1/m \}.$$ Therefore the set 	
\begin{eqnarray*}
F_{\rho, h_\rho} &:=& \bigcap_{m,n} F_{\rho, h_\rho, m, n}\\
&=&\{ p \in X\backslash X_0 : \| \rho(a_n)h_\rho \|^2 = \|\pi_p(a_n)h_p \|^2, ~ \forall~ n\in\bbN \}\\
&=&\{ p \in X\backslash X_0 : \| \rho(a)h_\rho \|^2 = \|\pi_p(a)h_p \|^2, ~ \forall~ a\in A \}
\end{eqnarray*}
is a measurable set. We show that $F_{\rho, h_\rho}=E_p$. To see this, simply note that all representations under consideration are irreducible and for $p\in F_{\rho, h_\rho}$ define the unitary $U_p$ by: $$U_p(\pi_p(a)h_p)=\rho(a)h_\rho,$$ for all $a\in A$. But then, $U_p=V_{\rho, p}$. So $\pi_p$ is equivalent to $\rho$ and this shows $F_{\rho, h_\rho} \subset E_p$. For the reverse inclusion, let $p \in E_p$. Then, by assumption we have $V_{\rho, p} \pi_p V_{\rho, p}^*=\rho$. But by the same vector equality $$V_{\rho, p}(\pi_p(a)h_p)=\rho(a)h_\rho,$$ we have $\| \rho(a)h_\rho \|^2 = \|\pi_p(a)h_p \|^2$ for all $a\in A$. This shows $E_p \subset F_{\rho, h_\rho}$ and thus $E_p=F_{\rho, h_\rho}$ is measurable. Finally, since the set $\{\rho\}$ is countable we get the countable partition $\coprod E_p$. 
	
Next, we prove the necessity condition. Now, let us assume that $\coprod E_p$ is a countable measurable partition. Let $\{\rho\}$ and $\{h_\rho\}_{\rho\in\{\rho\}}$ be as described earlier. As the partition is given in countably many sets we get the collection $\{\rho\}$ to be countable. Since $E_p$ is measurable for all $p\in X\backslash X_0$, the family of representations: $$\tilde{\pi}_p:=\rho$$ for $p\in E_p$ and $\rho$ corresponding to the class of $E_p$, is a measurable family of representations of $A$. So, we have the direct integral representation $$\tilde{\pi}=\di \tilde{\pi}_p$$ of $A$. Further, note that the function $p\mapsto \tilde{h}_p$ where $\tilde{h}_p:=h_\rho$ when $p\in E_p$ defines a measurable vector in $\di \tilde{H}_p$. Here $\tilde{H}_p =H_\rho$ when $p\in E_p$. Now, by construction $\pi_p$ is unitarily equivalent to $\tilde{\pi}_p$ for all $p\in X\backslash X_0$. By \cite[Theorem 8.28]{Tak1} we get a measurable operator function $p\mapsto U_p$ (defined everywhere by assigning unitaries on a set of measure zero if necessary) such that $U_p \pi_p U_p^*= \tilde{\pi}_p$ for all $p\in X\backslash X_0$ and the unitary $U=\di U_p$ implements the equivalence between $\pi$ and $\tilde{\pi}$. Now, the required family of unitaries, as in the statement, is simply given by $V_{\rho, p}=U_p$ for $p\in E_p$ and noting that the function $p\mapsto U_p^* \tilde{h}_p$ is a measurable vector in $\di H_p$ and we are done.
\end{proof}

A technical remark regarding Proposition \ref{mbl} is given below:

\begin{remark} \label{measure0}
The measure zero set $X_0$ is not unique in the disintegration of $\pi$ with respect to a fixed maximal abelian $\clL$. However, by uniqueness of the decomposition upto a set of measure zero (see \cite[Theorem 8.25]{Tak1}), we can demonstrate the unambiguity of our statement. Suppose $\tilde{X}_0$ is another such measure zero set in the disintegration of $\pi$ w.r.t. $\clL$. Let $$X \backslash \tilde{X}_0 = \coprod\tilde{E}_p,$$ where all notations are analogous to those defined earlier. Let $\tilde{\Lambda}$ denote the set of distinct classes of irreducible representations corresponding to $\{\tilde{E}_p\}$. Firstly, we note that (by adjusting the set of measure zero $\tilde{X}_0$ if necessary) both $\Lambda$ and $\tilde{\Lambda}$ have the same set of irreducible classes corresponding to $p\in X \backslash X_0 \cup \tilde{X}_0$ due to the uniqueness of the decomposition upto a set of measure zero. Let $\{\tilde{\rho}\}$ be an exhaustive set of representatives of $\tilde{\Lambda}$ and $\{\tilde{h}_{\tilde{\rho}}\}_{\tilde{\rho}\in\{\tilde{\rho}\}}$ be any tuple of unit vectors. Restricting to a subset of $\{\tilde{\rho}\}$ corresponding to $p\in X \backslash X_0 \cup \tilde{X}_0$, we consider a tuple $\{\beta\}$ from $\Lambda$ and unit vectors $\{h_\beta\}_{\beta\in \{\beta\}}$, corresponding to $p\in X \backslash X_0$, such that $\beta=\tilde{\rho}$ and $h_\beta=\tilde{h}_{\tilde{\rho}}$ for classes corresponding to $p\in X \backslash X_0 \cup \tilde{X}_0$. Now, assuming the condition in Proposition \ref{mbl} to be true, we obtain a family of unitaries $\{V^*_{\beta, p}: H_\beta \rightarrow H_p\}_{p\in E_p, \beta \in \{\beta\}}$ such that the function defined on $X\backslash X_0$ by $p\mapsto V^*_{\beta, p} h_\beta$ is a measurable vector and $V_{\beta, p} \pi_p V_{\beta, p}^*=\beta$ for all $p\in E_p$. Now restrict this measurable vector to $p\in X \backslash X_0 \cup \tilde{X}_0$ and suitably extend the restriction to $p\in X \backslash \tilde{X}_0$ by defining unitaries on the measure zero set $X_0 \backslash \tilde{X}_0$. So, we have obtained a family of unitaries $\{V^*_{\tilde{\rho}, p}: H_{\tilde{\rho}} \rightarrow H_p\}_{p\in \tilde{E}_p, \tilde{\rho} \in \{\tilde{\rho}\}}$ such that the vector function defined on $X\backslash \tilde{X}_0$ by $p\mapsto V^*_{\tilde{\rho}, p} h_{\tilde{\rho}}$ is measurable and $V_{\tilde{\rho}, p} \pi_p V_{\tilde{\rho}, p}^*=\tilde{\rho}$ for all $p\in \tilde{E}_p$. Thus the partition $X \backslash \tilde{X}_0 = \coprod\tilde{E}_p$ is measurable and our equivalent statements are unambiguous and valid, independent of the measure zero set $X_0$ considered for partitioning $X \backslash X_0$.  
\end{remark}

Now we are in a position to define a special class of maximal abelian \textit{decomposable} von-Neumann subalgebras of $\acomm$ by using \cite[Lemma 8.22]{Tak1}, Proposition \ref{mbl} and the discussions above:

\begin{definition} \label{dcmpL}
Let $\clL={\rm L}^\infty (X,\mu)$ be a maximal abelian von-Neumann subalgebra of $\acomm$. We say $\clL$ to be decomposable when the equivalent conditions of Proposition \ref{mbl} are true for the measure space $(X,\mu)$.
\end{definition}

Clearly, by \cite[Lemma 8.22]{Tak1} our definition of a \textit{decomposable} $\clL$ is independent of the choice of the standard measure space $(X,\mu)$. 

\begin{remark} \label{count}
Suppose $\pi:A \rightarrow B(H)$ be a representation of a separable unital C*-algebra $A$ on a separable Hilbert space $H$ such that $\acomm$ admits a decomposable maximal abelian von-Neumann subalgebra $\clL$. Then, by using Remark \ref{measure0} we can consider only those equivalence classes say $E_p$ for which $\mu(E_p)\neq 0$. Then by Proposition \ref{mbl}, we obtain a mutually orthogonal family of non-zero Hilbert space projections $\{P_\lambda\}_{\lambda\in \Lambda} \subset \clL$ where $P_\lambda=\chi_{E_p}\in {\rm L}^\infty (X,\mu)$ such that $E_p$ is the measurable equivalence class corresponding to the class of irreducible representation $\lambda\in \Lambda$. Clearly, this collection of mutually orthogonal projections is countable. We denote this distinguished family of projections by $\clP_\clL$.
\end{remark}

Next, we define decomposable completely positive maps for a given $\pi$ and maximal abelian von-Neumann subalgebra $\clL \subset \acomm$ as below:

\begin{definition} \label{dcp}
Let $\pi :A \rightarrow B(H)$ be a separable representation of a separable, unital C*-algebra $A$ such that $\acomm$ admit a decomposable maximal abelian von-Neumann subalgebra $\clL$. Let $\pi = \di \pi_p$ be the direct integral decomposition of $\pi$ w.r.t. $\clL$ and $\clR$ denote the von-Neumann algebra of decomposable operators on $\di H_p$. Let $\Phi : \clR \rightarrow \clR$ denote a completely positive map. Then, $\Phi$ is said to be decomposable if it satisfies:
$$\Phi (P_\lambda T)= P_\lambda \Phi (T)$$ for all $P_\lambda \in \clP_\clL$ and $T\in \clR$.
\end{definition}

Now we arrive at the notion of a decomposable weak expectation.

\begin{definition} \label{dwe}
Let $\pi :A \rightarrow B(H)$ be a separable representation of a separable, unital C*-algebra $A$ such that $\acomm$ admit a decomposable maximal abelian von-Neumann subalgebra $\clL$. Let $\clR$ denote the von-Neumann algebra of decomposable operators in $H$ w.r.t. $\clL$. A weak expectation $\Phi: \clR \rightarrow \adcom$ is said to be decomposable if it is also a decomposable completely positive map as in Definition \ref{dcp}.
\end{definition}

We emphasize here that, a decomposable weak expectation for a given representation (separable, unital) $\pi:A\rightarrow B(H)$ may only exist if the commutant of $\pi (A)$ admit a decomposable maximal abelian von-Neumann subalgebra $\clL$. We are now in a position to state the main theorem of this section.

\begin{theorem} \label{mt1}
Let $\pi :A \rightarrow B(H)$ be a separable representation of a separable, unital C*-algebra $A$ such that $\acomm$ admit a decomposable maximal abelian von-Neumann subalgebra $\clL$. Then, there exists a decomposable weak expectation of $\pi$ w.r.t. $\clL$ if and only if $\clP_\clL \subset \adcom$.
\end{theorem}

We defer the proof until the end of this section, but before that we record two important observations required for the proof of Theorem \ref{mt1}.

Let $(X,\mu)$ be a standard measure space. Let $\{\pi_p: A \rightarrow B(H_p)\}_{p\in X}$ be a measurable family of unital separable representations of $A$ such that, any two representations of the family are unitarily equivalent. Let $\sigma : A\rightarrow B(H_\sigma)$ be a representation such that $\pi_p \simeq \sigma$ for all $p\in X$. Let $\pi:=\di \pi_p$ denote the direct integral representation of $A$ on $H= \di H_p$. Then the von-Neumann algebra $\adcom$ is described as follows.

\begin{lemma} \label{dcom}
Let $\sigma :A \rightarrow B(H_\sigma)$ and $\pi:A\rightarrow B(H)$ be as described above. Then we have $$\pi(A)^{\prime \prime} = \left\{ \di U_p^*SU_p : S \in \sigma(A)^{\prime\prime}, p\mapsto U_p \right\}$$ where the function $\{p\mapsto U_p\}$ is a measurable unitary operator valued function such that $U_p^* \sigma U_p = \pi_p$ for a.e. $p\in X$.
\end{lemma}

\begin{proof}
First, note that, $\adcom \subset \di B(H_p)$ by construction. Let $$T=\di T_p \in \adcom ,$$ and $\|T\| = 1$. By the Kaplansky density theorem, the unit ball of $\pi (A)$ is strongly dense in the unit ball of $\adcom$. Using the separability of $H$, choose a sequence $\{ a_n \}$ of $A$ such that $\pi(a_n) \rightarrow T$ in the strong operator topology and $\| \pi(a_n)\| \leq 1$. Then, by \cite[Part II, Chapter 2, Proposition 4(i)]{DixV}, there exists a subsequence $\{a_{n_k}\}$ such that $\pi_p(a_{n_k}) \rightarrow T_p$ in strong topology for a.e $p\in X$. So evidently, by the unitary equivalence of $\pi_p$ and $\sigma$ by $U_p$ (see \cite[Theorem 8.28]{Tak1}), we have $\sigma (a_{n_k}) \rightarrow U_p T_p U_p^*$ for a.e. $p\in X$. By the uniqueness of the strong limit, we obtain a.e. $U_p T_p U_p^* = S \in \sigma(A)^{\prime\prime}$. Therefore, $T\in \adcom$ is of the form $\di U_p^*SU_p$ for some $S\in \sigma(A)^{\prime\prime}$ and a measurable function $p\mapsto U_p$ such that $U_p^* \sigma U_p = \pi_p$ for a.e. $p\in X$. 
	
Conversely, suppose $T=\di U_p^*SU_p$ as in the statement of the lemma. If $S=\sigma (a)$ for some $a\in A$, then by hypothesis, $T=\di \pi_p (a) \in \adcom$. Otherwise, suppose $S\in \sigma(A)^{\prime\prime}$ and $\|S\|=1$. By Kaplansky density and separability of $H_\sigma$, choose a sequence $\{a_n\} \subset A$ such that $\sigma(a_n) \rightarrow S$ in strong topology and $\| \sigma(a_n)\| \leq 1$. So we have $\pi_p (a_n) = U_p^* \sigma (a_n) U_p \rightarrow U_p^* S U_p$ in strong topology for a.e. $p\in X$. Let $T_n:= \di \pi_p (a_n)$, then $\| T_n \|\leq 1$ for all $n$ and $T_n (p)\rightarrow T_p=U_p^* S U_p$ strongly for a.e. $p\in X$. By \cite[Part II, Chapter 2, Proposition 4(ii)]{DixV} we have that $T_n \rightarrow T$ strongly and thus $T\in \adcom$. This completes the proof.
\end{proof}

Let $(X,\mu)$, $\{\pi_p\}$, $\pi$ and $\sigma$ be as stated before Lemma \ref{dcom}. We have the following result in regard to existence of weak expectations of $\pi$.

\begin{proposition} \label{intwep}
Let $(X,\mu)$, $\{\pi_p\}$, $\pi$ and $\sigma$ be as before. If $\sigma :A \rightarrow B(H_\sigma)$ admits a weak expectation then so does $\pi=\di \pi_p$.
\end{proposition}

\begin{proof}
Let $\Phi: B(H_\sigma)\rightarrow \sigma(A)^{\prime\prime}$ be a weak expectation of $\sigma$. By \cite[Theorem 8.28]{Tak1}, there exists a measurable unitary operator valued function $p\rightarrow U_p$ such that the unitary $U=\di U_p$ exhibits the unitary equivalence between $\pi$ and $\di \sigma$. Therefore$$U \left(\di B(H_p)\right)U^*=L^\infty (X,\mu) \bar{\otimes} B(H_\sigma).$$ Now, consider the completely positive map $$\omega \otimes {\rm id}: L^\infty (X,\mu) \bar{\otimes} B(H_\sigma) \rightarrow B(H_\sigma),$$ where $\omega \in L^1(X,\mu)$ is a state. Then the unital completely positive map $\Phi_\pi$ given by: $$T\mapsto U^*\left[\di (\Phi \circ (\omega \otimes {\rm id}))(UTU^*)\right]U$$ for $T\in\di B(H_p)$, is such that: $${\rm range~} \Phi_\pi \subset \adcom$$ by Lemma \ref{dcom} above, and $$\Phi_\pi (\pi(a)) = \pi(a)$$ for all $a\in A$, since $U\pi(a)U^*=1 \otimes \sigma (a) \in L^\infty (X,\mu) \bar{\otimes} B(H_\sigma)$. So, $$\Phi_\pi: \di B(H_p) \rightarrow \adcom$$ is a weak expectation of $\pi$ and our claim is proved.
\end{proof}

Now we can give the proof of the main theorem.

\subsection*{Proof of Theorem \ref{mt1}} The necessary condition simply follows from the definition of a decomposable weak expectation with respect to the maximal abelian von-Neumann subalgebra $\clL$. Let $\Phi: \clR \rightarrow \adcom$ be a decomposable weak expectation. Then, for $P_\lambda \in \clP_\clL$, $\Phi (P_\lambda)=P_\lambda \Phi (1)=P_\lambda \in \adcom$ and the necessity is proved.

Conversely, let $\clP_\clL \subset \adcom$. Let $\pi=\di \pi_p$ be the disintegration of $\pi$ with respect to $\clL$. Let $\{\rho_\lambda\}$ be a representative of $\Lambda$ and denote by $\{E_\lambda\}_{\lambda\in \Lambda}$ the measurable partition of $X\backslash X_0$ as before. Here, we have slightly changed notation to the effect that: $$E_\lambda :=\{ p\in X\backslash X_0| \pi_p \simeq \rho_\lambda\}$$ and $P_\lambda=\chi_{E_\lambda}$. By Remark \ref{count}, $\clP_\clL$ is a countable set and the strong sum $\sum P_\lambda$ is a well defined partition of unity in $\clL$. So we can write $\pi=\prod \pi_\lambda$ where $$\pi_\lambda:=P_\lambda \pi = \di_{E_\lambda} \pi_p$$ and by our hypothesis, $$\adcom=\prod \pi_\lambda (A)^{\prime\prime}.$$ By the definition of $E_\lambda$ and the fact that $\{\rho_\lambda\}$ are irreducible, we see that $\pi_\lambda$ admits a weak expectation, say $$\Phi_{\pi_\lambda} : \di_{E_\lambda} B(H_p) \rightarrow \pi_\lambda (A)^{\prime\prime}$$ by virtue of Proposition \ref{intwep}. Finally to conclude our proof, we simply note that the unital completely positive map $$\prod_{\lambda} \Phi_{\pi_\lambda} : \clR \rightarrow \adcom$$ is a decomposable weak expectation of $\pi$ on $\clR$. $\hfill\Box$\\

The motivation for the construction of a decomposable weak expectation stems from a desire to generalize Lance's construction of weak expectations for direct sums of representations, see \cite[Proposition 2.10]{L}. Our result describes a construction of a weak expectation in the direct integral case where a decomposable maximal abelian von-Neumann subalgebra of the commutant of the represented algebra exists.

We conclude this section by demonstrating that unitary equivalence respects the existence of a decomposable weak expectation. 

\begin{proposition} \label{ueq}
Let $A$ be a unital, separable C*-algebra and let $\pi_1$ and $\pi_2$ be unitarily equivalent non-degenerate representations of $A$ on Hilbert spaces $H_1$ and $H_2$ respectively. If $\pi_1(A)^\prime$ admit a decomposable maximal abelian von-Neumann subalgebra, then so does $\pi_2(A)^\prime$. Further, if $\pi_1$ admit a decomposable weak expectation, then so does $\pi_2$.
\end{proposition}

\begin{proof}
Let $U: H_1 \rightarrow H_2$ be the unitary such that $U\pi_1 U^*=\pi_2$. Let $\clL_1$ be a maximal abelian von-Neumann subalgebra of $\pi_1(A)^\prime$. Then $\clL_2:=U \clL_1 U^*$ is a maximal abelian von-Neumann subalgebra of $\pi_2(A)^\prime$. Now, consider the disintegration of $\pi_1$ and $\pi_2$ w.r.t. $\clL_1$ and $\clL_2$ respectively. Identifying the maximal abelian subalgebras with diagonalizable operators and composing the appropriate unitaries, we see that there is a unitary $$\tilde{U}: \di_{X_1} H_{1p} ~{\rm d} \mu_1 (p) \rightarrow \di_{X_2} H_{2q} ~{\rm d} \mu_2 (q)$$ such that $$\tilde{U}[ L^\infty (X_1,\mu_1)] \tilde{U}^* = L^\infty (X_2,\mu_2)$$ and $$\tilde{U} \left(\di_{X_1} \pi_{1p} ~{\rm d} \mu_1 (p)\right) \tilde{U}^* = \di_{X_2} \pi_{2q} ~{\rm d} \mu_2 (q).$$ By \cite[Part II, Chapter 6, Theorem 4]{DixV}, there exists a Borel isomorphism $\eta$ (modulo measure zero sets $Y_1$ and $Y_2$) from $X_1 \backslash Y_1$ onto $X_2 \backslash Y_2$ such that $$\tilde{U}=\di_{X_1 \backslash Y_1} V_p$$ where $V_p: H_{1p} \rightarrow H_{2\eta (p)}$ are unitaries such that $V_p \pi_{1p} V_p^* = \pi_{2 \eta(p)}$ for all $p\in X_1 \backslash Y_1$. In particular, if $0\neq\chi_{E} \in L^\infty (X_1,\mu_1)$ then $\tilde{U} \chi_{E} \tilde{U}^* = \chi_{\eta (E)} \in L^\infty (X_2,\mu_2)$ (upto a set of measure zero) and $\chi_{\eta (E)} \neq 0$.
	
Now, let $\clL_1$ be decomposable. Following the notations described earlier in this section, let $\chi_{E_\lambda} \in \clP_{\clL_1}$ for $\lambda \in \Lambda^1$. Using the family of unitaries $\{V_p\}$ we see that $\eta (E_\lambda)$ is a measurable equivalence class of $(X_2,\mu_2)$, $0 \neq \chi_{\eta({E_\lambda})} \in \clP_{\clL_2}$ and by slight abuse of notation (upto a set of measure zero) $\coprod \eta (E_\lambda)$ is a measurable partition of $X_2$. Note that, this partition corresponds precisely to the equivalence class partition of $$X_2\backslash Y_2=\coprod_{\lambda \in \Lambda^2_{\rm ess}} E^\prime_\lambda,$$ where $E^\prime_\lambda$ is the equivalence class of $q\in X_2\backslash Y_2$ such that $\pi_{2q}\simeq \lambda$. Clearly, by the argument above, we have $\Lambda^2=\Lambda^1$. This proves that, $\clL_2$ is decomposable. Then, by the unitary equivalence $\clL_2:=U \clL_1 U^*$, we see that if $\clP_{\clL_1}\subset \pi_1(A)^{\prime\prime}$ then $\clP_{\clL_2}\subset \pi_2(A)^{\prime\prime}$. Finally, an application of Theorem \ref{mt1} completes the proof.
\end{proof}

\section{Application: Decomposable Measures and the GNS Representation of a State}

The purpose of this section is to investigate the existence of decomposable weak expectations of the GNS representation of a given state of a separable unital C*-algebra. 

Let $A$ denote a separable unital C*-algebra and $\omega \in S(A)$. Then, every abelian von-Neumann subalgebra of $\pi_\omega(A)^{\prime}$ corresponds uniquely to a standard orthogonal probability measure $\mu$ defined on the Borel $\sigma$-algebra of $S(A)$ (by the separability assumption on $A$), such that, $\omega$ is the barycenter of $\mu$, i.e. $$\omega=\int_{S(A)} \omega^\prime~ {\rm d}\mu(\omega^\prime)$$ and by Theorem \ref{eff}, $\pi_\omega$ disintegrates as $$\pi_\omega=\di_{S(A)} \pi_{\omega^\prime}~ {\rm d}\mu(\omega^\prime)$$ with the measurable integrand $\omega^\prime \rightarrow \pi_{\omega^\prime}$, see \cite[Chapter 4]{OB1} for details.

Owing to the separability of $A$, the state space $S(A)$ is metrizable and therefore the set of extreme points of $S(A)$ is Borel measurable. The abelian von-Neumann subalgebra is maximal abelian if and only if the measure $\mu$ is maximal (see \cite[Theorem 4.2.2]{OB1})  and is supported on the set of pure states of $A$, denoted by $PS(A)$, which are precisely the set of extreme points of the state space $S(A)$.

Let $\omega \in S(A)$. Consider the standard probability measure space $(X,\mu)$ where $X={\overline{PS(A)}}^{wk^*}$ and $\mu$ is maximal orthogonal and has barycenter $\omega$ as described above. Further, let $\mu$ correspond to a maximal abelian von-Neumann subalgebra $\clL_\mu$ of $\pi_\omega(A)^{\prime}$ and $\pi_\omega$ disintegrates as above over $X$. 

Let $X_0:= X\backslash PS(A)$ and for $\omega^\prime \in PS(A)$, define the equivalence class: $$E_{\omega^\prime}:=\{\beta \in PS(A)| \pi_\beta \simeq \pi_{\omega^\prime}\}.$$ Clearly, $\mu (X_0)=0$ and $\clL_\mu$ is a \textit{decomposable} maximal abelian von-Neumann subalgebra if and only if $$PS(A)=\coprod_{\omega^\prime\in PS(A)} E_{\omega^\prime}$$ is a countable measurable partition.

Let us quickly recall that, since $\mu$ is an orthogonal measure, for any measurable subset $E$ of $X$, there exists a projection $P_E \in \pi_\omega (A)^{\prime}$ such that:$$\omega_E:=\int_E \omega^\prime \perp \int_{X\backslash E} \omega^\prime =: \omega_{X\backslash E}$$ and $\omega_E =\la \pi_\omega (\cdot)\xi_\omega, P_E \xi_\omega \ra$ and $\omega_{X\backslash E} =\la \pi_\omega (\cdot)\xi_\omega, (I-P_E) \xi_\omega \ra$, where $\xi_\omega \in H_\omega$ is the cylic vector for $\pi_\omega$. It is easy to see that, upto unitary identification, in the direct integral representation $P_E=\chi_E$. The measure $\mu$ is called \textit{disjoint} if $P_E\in \pi_\omega (A)^{\prime\prime}$ for \textit{any} measurable subset $E$ of $X$ i.e. if $\pi_{\omega_E}$ and $\pi_{\omega_{X\backslash E}}$ are disjoint representations. Disjoint measures correspond to subcentral decompositions of $\pi_\omega$, see \cite[Section 4.2.2]{OB1}.

Let $\omega_1$ and $\omega_2$ be two positive functionals on $A$ such that $\pi_{\omega_1}$ and $\pi_{\omega_2}$ are disjoint representations. We denote this by $\omega_1 \dot{\perp} \omega_2$. Now we define \textit{decomposable} measures.

\begin{definition} \label{dmes}
Let $\omega\in S(A)$ and $\mu$ be a maximal orthogonal measure such that $\omega$ is the barycenter of $\mu$. Then, call $\mu$ decomposable if $\clL_\mu$ is decomposable and $$\omega_{E_{\omega^\prime}} \dot{\perp} \omega_{X\backslash E_{\omega^\prime}}$$ for any $E_{\omega^\prime} \in \{E_{\omega^\prime}\}_{\omega^\prime \in S(A)}$ (which is a measurable family in that case).
\end{definition}

\begin{remark}
For the measure $\mu$ in Definition \ref{dmes} to be disjoint, the subalgebra $\clL_\mu$ needs to be subcentral. In contrast, our definition of decomposable measure requires $\clL_\mu$ to be decomposable as in Definition \ref{dcmpL} \textit{and} only the distinguished family of projections $\{P_{E_{\omega^\prime}}:=\chi_{E_{\omega^\prime}}\}$ corresponding to the measurable subsets $\{E_{\omega^\prime}\}_{\omega^\prime \in S(A)}$ to be in the center of $\pi_\omega(A)^{\prime\prime}$.
\end{remark}

The main theorem of this section is given below.

\begin{theorem} \label{mt2}
Let $\omega\in S(A)$. Then the GNS representation $\pi_\omega$ admits a decomposable weak expectation if and only if there exist a maximal orthogonal decomposable measure $\mu$ with $\omega$ as its barycenter.
\end{theorem}

\begin{proof}
The proof easily follows from Theorem \ref{mt1} and Definition \ref{dmes}. 
\end{proof}

We conclude this section with the following corollary of Theorem \ref{mt2}, which is a straightforward consequence of Proposition \ref{ueq}. The proof is omitted.

\begin{corollary}
Let $\omega\in S(A)$ and let $\mathcal{U}_\omega$ denote the unitary orbit of $\omega$ in the state space $S(A)$. If $\omega$ admits a decomposable measure then every $\tau \in \mathcal{U}_\omega$ admits a decomposable measure.
\end{corollary}

\section{An Example of a Non-degenerate Representation Admitting a Decomposable Weak Expectation}
In this section, we will see a prototype to construct an example of a representation of a unital separable C*-algebra $A$ which has a decomposable weak expectation. Suppose there exists a unital, separable C*-algebra $A$ with a sequence $(a_n)_{n \in \mathbb{N}}$ of elements of $A$ and a sequence of irreducible representations $(\pi_n)_{n \in \mathbb{N}}$ of $A$ on Hilbert space $(H_n)_{n \in \mathbb{N}}$ satisfying the following two conditions:

\begin{enumerate}
\item $a_n \neq a_m$ if $n \neq m$;
\item $\pi_n(a_m) = \delta_{n,m}$ 
\end{enumerate}

We claim that there exists a representation of $A$ which admits a decomposable weak expectation. We construct such a representation as follows: Consider the interval $[0, 1]$ with the usual Lebesgue measure, say $\mu$. Let $E_n = (\frac{1}{n-1}, \frac{1}{n}]$. Then each $E_n$ is a measurable set with $\mu(E_n) > 0$ and we get a countable measurable partition of $[0, 1]$ as $\coprod_{n \in \mathbb{N}} E_n \cup \{ 0 \}$. Now consider a decomposable representation 
\begin{equation*}
\pi := \di_{[0, 1]} \pi_p~ {\rm d}\mu(p) : A \longrightarrow \di_{[0, 1]} B(H_p) ~ {\rm d}\mu(p)
\end{equation*}
where $\pi_p = \pi_n$ and $H_p = H_n$ for $p \in E_n$. As $\pi_p$ is irreducible for almost all $p \in [0, 1]$, using Corollary \ref{max} we get the diagonalizable algebra $L^\infty \left ([0, 1], \mu \right )$ is a maximal abelian subalgebra of $\acomm$. Also from the construction of $\pi$ and by Definition \ref{dcmpL} it is clear that the maximal abelian subalgebra $L^\infty \left ([0, 1], \mu \right )$ is a decomposable maximal abelian subalgebra of $\acomm$. From the definition of $\pi$ we have 
\begin{equation*}
\pi(a_n) = \chi_{E_n}.
\end{equation*} 
Since $\chi_{E_n} \in \adcom$ for all $n \in \mathbb{N}$, by an application of Theorem \ref{mt1} we get a decomposable weak expectation of $\pi.$

Using the method described above, we construct an example of a non-degenerate representation of $\cfinfty$ admitting a decomposable weak expectation. 

\begin{example}
The full group C*-algebra $\cfinfty$ has a representation which admits a decomposable weak expectation.
\end{example}
\begin{proof}
We get the required representation of the C*-algebra $\cfinfty$ by factoring the $\cfinfty$ through the C*-algebra $C[0, 1]$ (C*-algebra of all continuous complex valued functions on $[0, 1]$). First we construct a family of irreducible representations of $C[0, 1]$ as follows: For $n \in \mathbb{N}$, define a pure state $\omega_{n} : C[0, 1] \rightarrow \mathbb{C}$ by $\omega_{n}(f) := f \left (\frac{1}{n} \right )$ for $f \in C[0, 1]$. Corresponding to $\omega_{n}$ we have the GNS triple $(\pi_{\omega_{n}}, H_{\omega_{n}}, \Omega_{\omega_{n}})$ and as $\omega_{n}$ is a pure state the representation $\pi_{\omega_{n}}$ is irreducible. Consider
\begin{eqnarray*}
N_{\omega_{n}} &=& \{f \in C[0, 1] \, : \, \omega_{n}(f^*f) = 0 \} \\
&=& \{f \in C[0, 1] \, : \, f^*f 
\left (\frac{1}{n} \right ) = 0 \} \\
&=& \{f \in C[0, 1] \, : \, f \left (\frac{1}{n} \right ) = 0 \}.
\end{eqnarray*} 
Then $N_{\omega_{n}}$ is a closed left ideal of $C[0, 1]$ and the map
\begin{eqnarray*}
(C[0, 1] / N_{\omega_{n}})^2 &\rightarrow& \mathbb{C} \\
(f + N_{\omega_{n}}, g + N_{\omega_{n}}) &:=& \omega_{n}(g^*f)
\end{eqnarray*}
is an inner product on $C[0, 1] / N_{\omega_{n}}$. The Hilbert space $H_{\omega_{n}}$ is the completion of $C[0, 1] / N_{\omega_{n}}$.

Construct a sequence functions $(f_n)_n$ in $C[0, 1]$ such that
\begin{equation*}
f_n \left (\frac{1}{m} \right ) = \delta_{n,m}.
\end{equation*}
We have a sequence of functions $(f_n)_n$ and a sequence of irreducible representations $(\pi_{\omega_{n}})_n$ of $C[0, 1]$ on $(H_{\omega_{n}})_n$. 
	
For any $g \in C[0, 1]$, 
\begin{equation*}
\pi_{\omega_{n}}(f_n)(g + N_{\omega_{n}}) = f_ng + N_{\omega_{n}}.
\end{equation*}
Since $(g-f_ng) \left (\frac{1}{n} \right ) = 0$, we get $g + N_{\omega_{n}} = f_ng + N_{\omega_{n}}$ and hence 
\begin{equation*}
\pi_{\omega_{n}}(f_n) = \mathrm{Id}_{H_{\omega_{n}}}.
\end{equation*}
Whereas if $n \neq m,$ then
\begin{equation*}
\pi_{\omega_{n}}(f_m)(g + N_{\omega_{n}}) = f_mg + N_{\omega_{n}}.
\end{equation*}
Since $f_mg\left (\frac{1}{n} \right ) = 0$, we get $f_mg + N_{\omega_{n}} = N_{\omega_{n}}$ and hence $\pi_{\omega_{n}}(f_m) = \mathrm{0}_{H_{\omega_{n}}}$. This implies
\begin{equation*}
\pi_{\omega_{n}}(f_m) = 0
\end{equation*} 

Let $\widetilde{\rho} : \cfinfty \rightarrow C[0, 1]$ be a surjective *-homomorphism (obtained by the universal property) and $(a_n)_n$ be a sequence of elements of $\cfinfty$ such that $\widetilde{\rho}(a_n) = f_n.$ Define an irreducible representation of $\cfinfty$ as 
\begin{equation*}
\rho_n := \pi_{\omega_{n}} \circ \widetilde{\rho} : \cfinfty \rightarrow B(H_{\omega_{n}}).
\end{equation*} 
Observe $\rho_n(a_m) = \pi_{\omega_{n}}(f_m) = \delta_{n,m}$.

Now consider the interval $[0, 1]$ with the usual Lebesgue measure $\mu$ and $E_n = (\frac{1}{n-1}, \frac{1}{n}]$ as defined earlier in this section. Define a decomposable representation of $\cfinfty$ as
\begin{equation*}
\rho := \di_{[0, 1]} \rho_p ~ {\rm d}\mu(p) : \cfinfty \longrightarrow \di_{[0, 1]} B(H_p) ~ {\rm d}\mu(p)
\end{equation*}
where $\rho_p = \rho_n$ and $H_p = H_{\omega_{n}}$ for $p \in E_n$. By using Corollary \ref{max}, Definition \ref{dcmpL} and irreducibilty of $\rho_p$ for almost all $p \in [0, 1]$, we get the diagonalizable algebra $L^\infty \left ([0, 1], \mu \right )$ is a decomposable maximal abelian subalgebra of $\rho(\cfinfty)^{\prime}$. Moreover, we have 
\begin{equation*}
\rho(a_n) = \chi_{E_n}.
\end{equation*} 
Hence $\chi_{E_n}$ belongs to $\rho(\cfinfty)^{\prime \prime}$ for all $n \in \mathbb{N}$. Then Theorem \ref{mt1} implies that the representation $\rho$ admits a decomposable weak expectation.

Now we take the help of Proposition \ref{intwep} to construct a decomposable weak expectation of $\rho$. Moreover, the construction gives not just a one but a family of decomposable weak expectations of $\rho$ as demonstrated below. We have
\begin{equation*}
\rho = \di_{[0, 1]} \rho_p ~ {\rm d}\mu(p) : \cfinfty \longrightarrow \di_{[0, 1]} B(H_p) ~ {\rm d}\mu(p)
\end{equation*}
is a non-degenerate representation of $\cfinfty$ where, $\rho_p = \rho_n$ and $H_p = H_{\omega_{n}}$ for $p \in E_n$. Since each $\rho_n$ is irreducible, by Arveson's extension we get a weak expectation of $\rho_n$ given by $\phi_n: B(H_{\omega_{n}}) \rightarrow \rho_n(\cfinfty)^{\prime\prime}.$ Now, consider the completely positive map 
\begin{equation*}
\tau_n \otimes {\rm id}: L^\infty (E_n, \mu) \bar{\otimes} B(H_{\omega_{n}}) \rightarrow B(H_{\omega_{n}}),
\end{equation*} 
where $\tau_n \in L^1(E_n,\mu)$ is a state. Let 
\begin{equation*}
\Phi_n : \di_{E_n} B(H_p) ~ {\rm d}\mu(p) \rightarrow \left ( \di_{E_n} \rho_p ~ {\rm d}\mu(p) \cfinfty \right)^{\prime \prime}
\end{equation*}  
be a unital completely positive map given by 
\begin{equation*}
T \mapsto \di_{E_n} (\phi_n \circ (\tau_n \otimes {\rm id}))(T)~ {\rm d}\mu(p)
\end{equation*}
for $T \in \di_{E_n} B(H_p) ~ {\rm d}\mu(p)$.
	
Then the unital completely positive map defined by
\begin{equation*}
\Phi_\rho := \prod_{n} \Phi_n : \di_{[0, 1]} B(H_p) ~ {\rm d}\mu(p) \rightarrow \rho \left ( \cfinfty \right)^{\prime \prime}
\end{equation*} 
is a decomposable weak expectation of $\rho$. 

The map $\Phi_\rho$ depends on the choice of states $\tau_n \in L^1(E_n,\mu)$ and the choice of weak expectations $\phi_n$ of $\rho_n.$ Therefore, by varying $\tau_n$ and $\phi_n$ we get a family of decomposable weak expectations of $\rho$ indexed by the states $\tau_n$ and an Arveson's extension of $\rho_n$. 
\end{proof}


\section*{Acknowledgement} 
The first named author is partially supported by Science and Engineering Board (DST, Govt. of India) grant no. ECR/2018/000016 and the second named author is supported by CSIR PhD scholarship award letter no. 09/1020(0142)/2019-EMR-I.

\subsection*{Declaration}
The authors declare that there is no conflict of interest.

\end{document}